\documentclass[11pt]{article}

\usepackage{amssymb,amsmath,amsfonts,amsthm}
\usepackage{latexsym}
\usepackage{graphics}
\usepackage{indentfirst}
\usepackage{lmodern}
\usepackage{hyperref}
\usepackage{mathrsfs}
\usepackage{array}
\usepackage{mathtools}
\usepackage{graphicx}
\usepackage{setspace}
\usepackage{enumerate}
\usepackage{tikz}
\usepackage{url,color}
\usepackage{chngcntr}
\usepackage{multicol}
\usepackage{comment}

\newtheorem{innercustomthm}{Theorem}
\newenvironment{customthm}[1]
  {\renewcommand\theinnercustomthm{#1}\innercustomthm}
  {\endinnercustomthm}
\newtheorem{innercustomlem}{Lemma}
\newenvironment{customlem}[1]
  {\renewcommand\theinnercustomlem{#1}\innercustomlem}
  {\endinnercustomlem}
\newtheorem{innercustomcor}{Corollary}
\newenvironment{customcor}[1]
  {\renewcommand\theinnercustomcor{#1}\innercustomcor}
  {\endinnercustomcor}
\newtheorem{innercustompro}{Proposition}
\newenvironment{custompro}[1]
  {\renewcommand\theinnercustompro{#1}\innercustompro}
  {\endinnercustompro}

\newtheorem*{thm*}{Theorem}
\newtheorem{thm}{Theorem}
\newtheorem{lem}[thm]{Lemma}
\newtheorem{pro}[thm]{Proposition}

\newtheorem{cor}[thm]{Corollary}
\newtheorem{conj}[thm]{Conjecture}

\newtheorem{ques}[thm]{Question}

\newcommand{\N}{\mathbb{N}}

\newcommand{\Aut}{\mathrm{Aut}}

\begin{document}

\title{Counting List Colorings of Unlabeled Graphs}

\author{Hemanshu Kaul~\thanks{Department of Applied Mathematics, Illinois Institute of Technology, Chicago, IL, USA (kaul@illinoistech.edu)}
\and 
Jeffrey A. Mudrock~\thanks{Department of Mathematics and Statistics, University of South Alabama, Mobile, AL, USA (mudrock@southalabama.edu)}}

\maketitle

\begin{abstract}
The classic enumerative functions for counting colorings of a graph $G$, such as the chromatic polynomial $P(G,k)$, do so under the assumption that the given graph is labeled. In 1985, Hanlon defined and studied the chromatic polynomial for an unlabeled graph $\mathcal{G}$, $P(\mathcal{G}, k)$. Determining $P(\mathcal{G}, k)$ amounts to counting colorings under the action of automorphisms of $\mathcal{G}$.  
In this paper, we consider the problem of counting list colorings of unlabeled graphs.  We extend Hanlon's definition to the list context and define the unlabeled list color function, $P_\ell(\mathcal{G}, k)$, of an unlabeled graph $\mathcal{G}$. In this context, we pursue a fundamental question whose analogues have driven much of the research on counting list colorings and its generalizations: For a given unlabeled graph $\mathcal{G}$, does $P_\ell(\mathcal{G}, k)  =  P(\mathcal{G}, k)$ when $k$ is large enough?  We show the answer to this question is yes for almost all graphs, in particular, for a large class of unlabeled graphs that includes point-determining graphs (also known as twin-free graphs, irreducible graphs, and mating graphs).

\medskip



\end{abstract}

\section{Introduction}\label{intro}

Counting colorings of graphs has a long history going back to Birkhoff~\cite{B12} and his introduction of the chromatic polynomial ($P(G,k)$, the number of proper $k$-colorings of a graph $G$) in 1912 with the hope of using it to make progress on the four color problem. Since then the chromatic polynomial and its generalizations have become central objects of study in Algebraic and Enumerative Combinatorics. Beyond classic  coloring, the notion of counting functions has been widely studied for list coloring, which was introduced independently by Vizing~\cite{V76} and Erd\H{o}s, Rubin, and Taylor~\cite{ET79} in the 1970s, and DP (or correspondence) coloring which was introduced by Dvo\v{r}\'{a}k and Postle~\cite{DP15} in 2015. 

Typically when classical colorings, list colorings, or DP colorings of a graph are counted it is assumed that the graph is labeled. 
In 1985, Hanlon defined and studied the chromatic polynomial for an unlabeled graph $\mathcal{G}$, $P(\mathcal{G}, k)$. Hanlon showed that determining $P(\mathcal{G}, k)$ amounts to counting colorings under the action of automorphisms of $\mathcal{G}$.  Hanlon's work has been generalized in a couple of directions. One being the ``orbital chromatic polynomial'' defined by considering any subgroup of Aut(G), that still considers classic graph coloring but generalizes the context of equivalence between such colorings (see~\cite{CK07}). In \cite{D19}, Hanlon's work is extended to counting colorings of unlabeled signed graphs. 

In this paper, we pursue a different perspective and consider the problem of counting list colorings of unlabeled graphs. Unlike previous works, the challenge here is that Burnside's Lemma/ Orbit Counting Lemma is of limited utility in this context. We give a natural definition of the unlabeled list color function, $P_\ell(\mathcal{G}, k)$, of an unlabeled graph $\mathcal{G}$ that counts its non-equivalent list colorings guaranteed over all $k$-list assignments. 

When Kostochka and Sidorenko introduced the list color function in 1990 ($P_\ell(G,k)$, the guaranteed number of list colorings of a labeled graph $G$ over all list assignments of $k$ colors), the first question they asked was whether $P_\ell(G,k)$ equals $P(G,k)$ when $k$ is large enough. This fundamental question has motivated much research on enumerative aspects of list coloring~\cite{DZ22} and DP-coloring~\cite{KM19}. Starting with Donner~\cite{D92} in 1992, and follow-up papers \cite{DZ22, T09, WQ17}, it is now known that $P_\ell(G,k) = P(G,k)$  when $k\ge |E(G)|-1$. This phenomenon of the list color function of a graph equaling the corresponding chromatic polynomial when the number of colors is large enough has also been pursued recently in the context of counting list packings of a labeled graph $G$ (list packing asks for the existence of multiple pairwise disjoint list colorings of the graph)~\cite{KM24}. However, note that this phenomenon need not always be true for all graphs, for example, in the context of DP colorings of graphs and the associated DP color function vs the chromatic polynomial (see~\cite{KM19}).

Our primary objective in this paper is to study this phenomenon in context of list colorings of unlabeled graphs: Does $P_\ell(\mathcal{G}, k)  =  P(\mathcal{G}, k)$ for a given unlabeled graph $\mathcal{G}$ when $k$ is large enough? We show that this is true for a large class of graphs that includes point-determining graphs. Point-determining graphs (also referred in literature as twin-free graphs, irreducible graphs, and mating graphs) have long been studied since being first considered by Sabidussi~\cite{S61} in 1961. It is known that almost all graphs are point-determining. We also show how to generate graphs that are not point-determining but satisfy this property.

\subsection{List Coloring and Counting List Colorings} \label{list}

In this paper all graphs are nonempty, finite, simple graphs.  Except that in Theorem~\ref{thm: transposition} below we permit $G$ to be the empty graph, with the usual conventions $P(G,k)=1$ and $|Aut(G)|=1$.  Generally speaking, we follow West~\cite{W01} for terminology and notation.  The set of natural numbers is $\N = \{1,2,3, \ldots \}$.  For $m \in \N$, we write $[m]$ for the set $\{1, \ldots, m \}$.  We denote the automorphism group of a graph $G$ by $\Aut(G)$.  When $\pi$ is a permutation of a nonempty set $A$, we say $\pi = id$ when $\pi$ is the identity permutation on $A$.  Also, any permutation $\pi$ has a unique decomposition into disjoint cycles (up to the order of the cycles).  If we say $C_1 \ldots C_s$ is the cycle decomposition of $\pi$, then $C_1 \ldots C_s$ is a particular ordering of the unique decomposition of $\pi$ into disjoint cycles that includes all 1-cycles.  Strictly speaking $C_i$ is a function for each $i \in [s]$; however, when convenient in context, we will sometimes view $C_i$ as a subset of $A$ that consists of all the elements of $A$ appearing in the cycle $C_i$.  Additionally, we use $K_n$ for the complete graphs on $n$ vertices, and we use $P_n$ for the paths on $n$ vertices.  If $G$ and $H$ are vertex disjoint graphs, we write $G \vee H$ for the join of $G$ and $H$.

In the classical vertex coloring problem, we wish to color the vertices of a (labeled) graph $G$ with up to $k$ colors from $[k]$ so that adjacent vertices receive different colors, a so-called \emph{proper $k$-coloring}.  The \emph{chromatic number} of a graph, denoted $\chi(G)$, is the smallest $k$ such that $G$ has a proper $k$-coloring.  List coloring is a well-known variation on classical vertex coloring which was introduced independently by Vizing~\cite{V76} and Erd\H{o}s, Rubin, and Taylor~\cite{ET79} in the 1970s.  For list coloring, we associate a \emph{list assignment} $L$ with a graph $G$ such that each vertex $v \in V(G)$ is assigned a set of available colors $L(v)$.  We say $G$ is \emph{$L$-colorable} if there is a proper coloring $f$ of $G$ such that $f(v) \in L(v)$ for each $v \in V(G)$.  We refer to $f$ as a \emph{proper $L$-coloring} of $G$.  A list assignment $L$ for $G$ is called a \emph{$k$-assignment} if $|L(v)|=k$ for each $v \in V(G)$.  The \emph{list chromatic number} of a graph $G$, denoted $\chi_\ell(G)$, is the smallest $k$ such that $G$ is $L$-colorable whenever $L$ is a $k$-assignment for $G$.  It is immediately obvious that for any graph $G$, $\chi(G) \leq \chi_\ell(G)$.

In 1912 Birkhoff~\cite{B12} introduced the notion of the chromatic polynomial with the hope of using it to make progress on the four color problem.  For $k \in \N$, the \emph{chromatic polynomial} of a graph $G$, $P(G,k)$, is the number of proper $k$-colorings of $G$. It is well-known that $P(G,k)$ is a monic polynomial in $k$ of degree $|V(G)|$ (e.g., see~\cite{B94}).  For example, $P(K_n,k) = \prod_{i=0}^{n-1} (k-i)$, and for any graph $G$, $P(K_n \vee G,k) = \left(\prod_{i=0}^{n-1} (k-i) \right) P(G, k-n)$.  

The notion of chromatic polynomial was extended to list coloring in the early 1990s by Kostochka and Sidorenko~\cite{KS90}.  If $L$ is a list assignment for $G$, we use $P(G,L)$ to denote the number of proper $L$-colorings of $G$. The \emph{list color function} $P_\ell(G,k)$ is the minimum value of $P(G,L)$ where the minimum is taken over all possible $k$-assignments $L$ for $G$.  Since a $k$-assignment could assign the same $k$ colors to every vertex of $G$, $P_\ell(G,k) \leq P(G,k)$ for each $k \in \N$.  In general, the list color function can differ significantly from the chromatic polynomial for small values of $k$.  However, in 1992, answering a question of Kostochka and Sidorenko~\cite{KS90}, Donner~\cite{D92} showed that for any graph $G$ there is an $N \in \N$ such that $P_\ell(G,k) = P(G,k)$ whenever $k \geq N$. Dong and Zhang~\cite{DZ22} (improving upon results in~\cite{D92}, \cite{T09}, and~\cite{WQ17}) subsequently showed the following.
\begin{thm} [\cite{DZ22}] \label{thm: fenming}
Let $G $ be a simple graph with $n$ vertices and $m \geq 4$ edges. Then, for any $k$-assignment $L$ of $G$ with $k \ge m - 1$,
$$P(G, L) - P(G, k) \ge \left((k - m + 1)k^{n-3} + (k - m + 3)(c/3) k^{n-5}\right)
\sum_{uv\in E(G)}|L(u) - L(v)|,$$
where $c \ge (m-1)(m-3)/8$, and particularly, when $G$ is $K_3$-free, $c \ge {\binom{m-2}{2}} + 2\sqrt{m} - 3.$
\end{thm}
Our main motivation for this paper is to prove an analogue of the following result in the context of unlabeled graphs:  For any graph $G$, $P_{\ell}(G,k)=P(G,k)$ whenever $k$ is a positive integer satisfying $k \geq |E(G)|-1$.~\footnote{This follows from Theorem~\ref{thm: fenming} for graphs with at least four edges and is easily checked for graphs with smaller size.}    

\subsection{Coloring Unlabeled Graphs}\label{unlabeled}

In 1985, Hanlon~\cite{H85} defined and studied an extension of the chromatic polynomial to unlabeled graphs. 

Suppose $\mathcal{G}$ is an equivalence class of graphs under the isomorphism relation and each element in $\mathcal{G}$ has vertex set $\{v_1, \ldots, v_n \}$.  We refer to $\mathcal{G}$ as an \emph{unlabeled graph of order $n$}.  Let $S_n$ be the symmetric group on $\{v_1, \ldots, v_n \}$.  For $\pi \in S_n$ and $G \in \mathcal{G}$, $\pi G$ is the element of $\mathcal{G}$ with edge set $\{ \pi(v_i)\pi(v_j) : v_iv_j \in E(G) \}$.  Also, note that $\pi$ is an \emph{automorphism} of $G$ if and only if $\pi G = G$.  Let 
$$u(\mathcal{G},k) = \{ (G, f) : \text{$G \in \mathcal{G}$ and $f$ is a proper $k$-coloring of $G$} \}.$$
Notice that $A: S_n \times u(\mathcal{G},k) \rightarrow u(\mathcal{G},k)$ given by $A(\pi, (G,f)) = (\pi G, f \pi^{-1})$ is a group action of $S_n$ on $u(\mathcal{G},k)$.  A \emph{proper $k$-coloring of $\mathcal{G}$} is an orbit of $S_n$ acting on $u(\mathcal{G},k)$.  We let $P(\mathcal{G}, k)$ denote the number of proper $k$-colorings of $\mathcal{G}$ for each $k \in \N$.

Now, suppose that $G \in \mathcal{G}$ and $\pi \in S_n$ is an automorphism of $G$.  A \emph{proper $(\pi,k)$-coloring} is a proper $k$-coloring, $f$, of $G$ with the property that $f(\pi(v)) = f(v)$ for each $v \in V(G)$.  Let $P(G, \pi, k)$ be the number of proper $(\pi,k)$-colorings of $G$.  Hanlon showed the following.
\begin{thm}[\cite{H85}] \label{thm: Hanlon}
    For any $k \in \N$ and $G \in \mathcal{G}$, $$P(\mathcal{G},k) = \frac{1}{|\Aut(G)|} \sum_{\pi \in \Aut(G)} P(G, \pi, k).$$
\end{thm}
It is worth noting that one of the keys to the proof of Theorem~\ref{thm: Hanlon} is Burnside's Lemma.  We now mention an important fact about $P(G, \pi, k)$ for $\pi \in \Aut(G)$.  Suppose that $\pi \in \Aut(G)$ and $C_1 \ldots C_s$ is the cycle decomposition of $\pi$.  The \emph{quotient of $G$ with respect to $\pi$}, denoted $G : \pi$, is the graph with vertex set $\{C_1, \ldots, C_s\}$ and edges created so that $C_iC_j \in E(G : \pi)$ if and only if there is a $u \in C_i$ and $v \in C_j$ such that $uv \in E(G)$.

\begin{lem} [\cite{H85}] \label{lem: reduce}
Suppose $G$ is a graph with $V(G) = \{v_1, \ldots, v_n \}$ and $k \in \N$.  Also, suppose that $\pi \in \Aut(G)$ and $C_1 \ldots C_s$ is the cycle decomposition of $\pi$.  Then the following statements hold.
\\
(i)  If there is an $i \in [s]$ such that $C_i$ contains two adjacent vertices in $G$, then $P(G, \pi, k) = 0$.
\\
(ii)  Otherwise, $P(G, \pi, k) = P(G : \pi, k)$. 
\end{lem} 
Since the identity permutation is in $\Aut(G)$, Theorem~\ref{thm: Hanlon} and Lemma~\ref{lem: reduce} imply that $P(\mathcal{G},k)$ is a polynomial in $k$ of degree $n$ with leading coefficient $1/ |\Aut(G)|$.  With this in mind, we refer to $P(\mathcal{G},k)$ as the \emph{unlabeled chromatic polynomial} of $\mathcal{G}$.

\subsection{List Coloring Unlabeled Graphs}

Suppose that $G \in \mathcal{G}$ and $L$ is a $k$-assignment for $G$.  Suppose $U(G,L)$ is the set of proper $L$-colorings of $G$.  We define an equivalence relation $\sim$ on $U(G,L)$ as follows.  For $f, g \in U(G,L)$, suppose $f \sim g$ if there is a $\pi \in \Aut(G)$ such that $f \pi = g$.  We refer to $f$ and $g$ as \emph{equivalent proper $L$-colorings of $G$}.  Let $u_{\ell}(G,L)$ be the number of equivalence classes of $\sim$.

We let $P_{\ell}(\mathcal{G},k)$ be the minimum value of $u_{\ell}(G,L)$ where $G$ is an arbitrarily chosen member of $\mathcal{G}$ and the minimum is taken over all possible $k$-assignments, $L$, for $G$.  We refer to $P_{\ell}(\mathcal{G},k)$ as the \emph{unlabeled list color function} of $\mathcal{G}$.  

We now show the unlabeled list color function of $\mathcal{G}$ is a natural extension of the unlabeled chromatic polynomial of $\mathcal{G}$ to the list context.  Suppose that $C$ is the $k$-assignment for $G$ that assigns $[k]$ to each vertex of $G$.  Notice that $A: \Aut(G) \times U(G,C) \rightarrow U(G,C)$ given by $A(\pi, f) = f \pi^{-1}$ is a group action of $\Aut(G)$ on $U(G,C)$.  Now, suppose for each $\pi \in \Aut(G)$ we let $\psi(\pi) = |\{f \in U(G,C) : f \pi^{-1} = f \}|$.  Burnside's Lemma gives a formula for counting the number of orbits of our group action, $u_{\ell}(G,C)$, in terms of $\psi$~(see~\cite{W20}). Specifically,
$$u_{\ell}(G,C) = \frac{1}{|\Aut(G)|} \sum_{\pi \in \Aut(G)} \psi(\pi) = \frac{1}{|\Aut(G)|} \sum_{\pi \in \Aut(G)} P(G, \pi, k) = P(\mathcal{G},k).$$
So, we clearly have $P_{\ell}(\mathcal{G},k) \leq P(\mathcal{G},k)$ for each $k \in \N$.  With this in mind, the following question is now a natural extension of the original question of Kostochka and Sidorenko~\cite{KS90} for labeled graphs.
\begin{ques} \label{ques: fundamental2}
For which unlabeled graphs $\mathcal{G}$ does there exist an $N \in \N$ so that $P_{\ell}(\mathcal{G},k) = P(\mathcal{G},k)$ whenever $k \geq N$?
\end{ques}
In this paper we make progress on Question~\ref{ques: fundamental2}; however, in general it remains open.  In fact, it is possible that all unlabeled graphs have the property described in Question~\ref{ques: fundamental2}. 

\subsection{Outline of the Paper}\label{outline}

We now present an outline of the paper.  In Section~\ref{prelim} we make some fundamental observations that are used in the remainder of the paper.  In particular, we partially extend Theorem~\ref{thm: Hanlon} to the list context.  Suppose $G$ is a graph, $L$ is a $k$-assignment for $G$, and $\pi \in \Aut(G)$.  If $P(G, \pi, L)$ is the number of proper $L$-colorings $f$ of $G$ with the property that $f(\pi(v)) = f(v)$ for each $v \in V(G)$, then we have the following result. 

\begin{lem} \label{lem: ineq}
Suppose $L$ is a $k$-assignment for $G \in \mathcal{G}$. Then, 
$$u_{\ell}(G,L) \ge \frac{1}{|\Aut(G)|} \sum\limits_{\pi \in \Aut(G)} P(G, \pi, L).$$
\end{lem}

The inequality in Lemma~\ref{lem: ineq} can not be replaced with equality for all graphs $G$.  Indeed, if $G$ is an edgeless graph on $n$ vertices with $n \geq 2$ and $L$ is a $k$-assignment for $G$ such that $L(u) \cap L(v) = \emptyset$ whenever $u, v \in V(G)$ and $u \neq v$, then $u_{\ell}(G,L) = k^n$, but $\frac{1}{|\Aut(G)|} \sum\limits_{\pi \in \Aut(G)} P(G, \pi, L) = k^n/n!$. We end Section~\ref{prelim} by considering unlabeled list color functions of disconnected, unlabeled graphs.  We prove the following.

\begin{pro} \label{pro: disjoint}
    Suppose $G_1$ and $G_2$ are connected graphs that are non-isomorphic and vertex disjoint. Let $G$ be the disjoint union of $G_1$ and $G_2$. Suppose $\mathcal{G}$, $\mathcal{G}_1$, and $\mathcal{G}_2$ are the unlabeled graphs corresponding to $G$, $G_1$ and $G_2$, respectively.  Then, $P_{\ell}(\mathcal{G},k) = P_{\ell}(\mathcal{G}_1,k)\;P_{\ell}(\mathcal{G}_2,k)$.
\end{pro}

Importantly, notice that Proposition~\ref{pro: disjoint} says nothing about unlabeled disconnected graphs that have two distinct, isomorphic, connected components.  In fact, the formulas at the end of Proposition~\ref{pro: disjoint} do not hold in this scenario.  For example, one can use Theorem~\ref{thm: Hanlon} to show that if $\overline{K}_n$ is an unlabeled, edgeless $n$-vertex graph with $n \geq 2$, then for each $k \in \N$, $$P(\overline{K}_n,k) = \frac{1}{n!} \prod_{i=0}^{n-1} (k+i) \neq k^n = P(\overline{K}_1,k)^n.$$
This observation along with Theorem~\ref{thm: transposition} below implies that there is a $k \in \N$ such that $P_{\ell}(\overline{K}_2,k) \neq P_{\ell}(\overline{K}_1,k)^2 $.  

Our main focus in Section~\ref{proof} is connected graphs, but studying the unlabeled list color functions of disconnected graphs with at least two isomorphic connected components naturally leads to further questions.  In fact, we conjectured the following.

\begin{conj} \label{conj: provoke}
Let $\overline{K}_n$ be an unlabeled, edgeless $n$-vertex graph. For each $n \in \N$ there is an $N \in \N$ such that $P_{\ell}(\overline{K}_n,k) = P(\overline{K}_n,k)$ whenever $k \geq N$.
\end{conj}

Corollary~\ref{cor: pointdetermine} and Theorem~\ref{thm: transposition} imply the truth of Conjecture~\ref{conj: provoke} when $n \in [2]$.  Conjecture~\ref{conj: provoke} was subsequently proven in the affirmative in~\cite{KMS26}.  In fact, it is shown that $P_{\ell}(\overline{K}_n,k) = P(\overline{K}_n,k)$ for all $n,k \in \N$.  Furthermore, it is shown in~\cite{KMS26} that a disconnected graph has the property described in Question~\ref{ques: fundamental2} if all of its connected components have this property.

In Section~\ref{proof} we show that all unlabeled point-determining graphs have the property described in Question~\ref{ques: fundamental2}.  A graph $G$ is called \emph{point-determining} if no two distinct vertices in $G$ have the same neighborhood in $G$.  Point-determining graphs (also referred in literature as twin-free graphs, irreducible graphs, and mating graphs) have long been studied in the context of various graph colorings, graph homomorphisms, graph domination, and mating systems (see e.g.,~\cite{GL11, S73}) since being first considered by Sabidussi~\cite{S61} in 1961. Unlabeled point-determining (connected) graphs have been enumerated, see~\cite{GL11, K07, R89} and sequences \emph{A004110} and \emph{A004108} in~\cite{S}. Note that asymmetric graphs are point-determining, and it is known that almost all graphs are asymmetric (see~\cite{ER63}).

We prove the following theorem using Theorem~\ref{thm: fenming} and Lemma~\ref{lem: ineq}.

\begin{thm} \label{thm: usefenmingnminus2}
Suppose $\mathcal{G}$ is an unlabeled, connected graph of order $n$ and size $m \geq 4$ with $G \in \mathcal{G}$.  Suppose for each $\pi \in \Aut(G) - \{id \}$ with cycle decomposition $C_1 \ldots C_s$ we have: If for each $i \in [s]$ the vertices in $C_i$ are pairwise nonadjacent in $G$, $s \leq n-2$.  Then, there exists an $N \in \N$ such that $P_{\ell}(\mathcal{G},k) = P(\mathcal{G},k)$ whenever $k \geq N$.
\end{thm}

After we handle  unlabeled, connected, point-determining graphs of size at most 3, Theorem~\ref{thm: usefenmingnminus2} allows us to prove the following.

\begin{cor} \label{cor: pointdetermine}
If $\mathcal{G}$ is an unlabeled, connected, point-determining graph, then there exists an $N \in \N$ such that $P_{\ell}(\mathcal{G},k) = P(\mathcal{G},k)$ whenever $k \geq N$.
\end{cor}

It is natural to ask whether we can extend this theorem beyond the class of point-determining graphs. We show that the operation of taking the join of an appropriate point-determining graph (e.g., complete graphs and asymmetric graphs) with two nonadjacent vertices will preserve this property. Note the new graph we get by this operation is not point-determining.

\begin{thm} \label{thm: transposition}
Suppose $\mathcal{G}$ is an unlabeled graph of order $n$ with $n$ possibly zero. Suppose $G \in \mathcal{G}$.  Suppose that $G$ is point-determining and for each $\pi \in \Aut(G) - \{id\}$ if $C_1 \ldots C_s$ is the cycle decomposition of $\pi$, then there is an $i \in [s]$ such that $C_i$ contains at least two adjacent vertices. 

Let $G'$ be the graph obtained from $G$ by adding nonadjacent vertices $x$ and $y$ and edges so that $x$ and $y$ are adjacent to every element of $V(G)$.  Also, let $\mathcal{G}'$ be the unlabeled graph such that $G' \in \mathcal{G}'$.  Then, there exists an $N \in \N$ such that $P_{\ell}(\mathcal{G}',k) = P(\mathcal{G}',k)$ whenever $k \geq N$.
\end{thm}

This theorem gives evidence that the answer to Question~\ref{ques: fundamental2} may very well be all unlabeled graphs.

\section{Preliminary Results}\label{prelim}

We start with a useful observation.

\begin{pro} \label{pro: doug}
Suppose $L$ is a $k$-assignment for graph $G$.  If $f$ and $g$ are proper $L$-colorings of $G$ such that there is a $\sigma \in \Aut(G)$ satisfying $f \sigma = g$, then
$$|\{ \pi \in \Aut(G) : f \pi = g \}| = |\{ \pi \in \Aut(G) : f \pi = f \}|.$$
\end{pro}

\begin{proof}
Let $A = \{ \pi \in \Aut(G) : f \pi = g \}$ and $B = \{ \pi \in \Aut(G) : f \pi = f \}$.  Let $M_1 : B \rightarrow A$ be given by $M_1(\pi) = \pi \sigma$, and let $M_2: A \rightarrow B$ be given by $M_2(\pi) = \pi \sigma^{-1}$.  The result follows from the fact that $M_1$ and $M_2$ are injections.
\end{proof}

Next, we define a generalization of proper $(\pi,k)$-coloring and use it to give a lower bound on $P_{\ell}(\mathcal{G},k)$. Suppose that $G \in \mathcal{G}$ and $\pi \in \Aut(G)$. Let $L$ be a $k$-assignment of $G$. A \emph{proper $(\pi,L)$-coloring} is a proper $L$-coloring, $f$, of $G$ with the property that $f(\pi(v)) = f(v)$ for each $v \in V(G)$.  Let $P(G, \pi, L)$ be the number of proper $(\pi,L)$-colorings of $G$. We now partially extend Theorem~\ref{thm: Hanlon} to the list context.

\begin{customlem}{\bf\ref{lem: ineq}}
Suppose $L$ is a $k$-assignment for $G \in \mathcal{G}$. Then, $$u_{\ell}(G,L) \ge \frac{1}{|\Aut(G)|} \sum\limits_{\pi \in \Aut(G)} P(G, \pi, L)$$.
\end{customlem}

\begin{proof}
    Let $\mathcal{E}$ be the set of all equivalence classes counted by $u_{\ell}(G,L)$, and for each $E \in \mathcal{E}$, let $f_E$ be a representative of the equivalence class $E$.  Then,
\begin{align*}
u_{\ell}(G,L) &= \frac{1}{|\Aut(G)|} \sum_{E \in \mathcal{E}} |\Aut(G)| \\
 &\geq \frac{1}{|\Aut(G)|} \sum_{E \in \mathcal{E}} \sum_{g \in E} |\{ \pi \in \Aut(G) : g \pi = f_E \}| \\
 &= \frac{1}{|\Aut(G)|} \sum_{E \in \mathcal{E}} \sum_{g \in E} |\{ \pi \in \Aut(G) : g \pi = g \}|\\
&= \frac{1}{|\Aut(G)|} \sum_{\pi \in \Aut(G)} P(G, \pi, L).
\end{align*}

\end{proof}

Next, we prove a generalization of Lemma~\ref{lem: reduce} to the list coloring context.

\begin{lem} \label{lem: reduce-list}
Suppose $G$ is a graph with $V(G) = \{v_1, \ldots, v_n \}$.  Also, suppose that $\pi \in \Aut(G)$ and $C_1 \ldots C_s$ is the cycle decomposition of $\pi$.  If $L$ is a $k$-assignment for $G$, then the following statements hold.
\\
(i)  If there is an $i \in [s]$ such that $C_i$ contains two adjacent vertices in $G$, then $P(G, \pi, L) = 0$.
\\
(ii)  Otherwise, $P(G, \pi, L) = P(G : \pi , L')$ where $L'$ is the list assignment for $G : \pi$ given by $L'(C_i) = \bigcap_{v \in C_i} L(v)$ for each $i \in [s]$. 
\end{lem} 

\begin{proof}
For statement (i), notice that in a proper $(\pi, L)$-coloring of $G$, each vertex in $C_i$ is colored with the same color.

For statement (ii), suppose $\mathcal{A}$ is the set of proper $(\pi,L)$-colorings of $G$, and suppose $\mathcal{B}$ is the set of proper $L'$-colorings of $G : \pi$.  It is easy to check that $\mathcal{A}$ is empty if and only if $\mathcal{B}$ is empty.  So, we may assume $\mathcal{A}$ is nonempty.  Let $M: \mathcal{A} \rightarrow \mathcal{B}$ be given by $M(f) = g$ where $g$ is the $L'$-coloring of $G : \pi$ obtained by coloring $C_j$ with $f(v)$ where $v$ is an arbitrarily chosen element of $C_j$ for each $j \in [s]$.  Notice that $M$ is well-defined since $f$ must color every vertex of a cycle of $\pi$ with the same color.  It is easy to verify that $M$ is a bijection.  
\end{proof}

We now give an application of Lemmas~\ref{lem: ineq} and~\ref{lem: reduce-list} which will be useful in the proof of Corollary~\ref{cor: pointdetermine}.

\begin{pro} \label{pro: oneaut}
Suppose $\mathcal{G}$ is an unlabeled graph and $G \in \mathcal{G}$.  Suppose that for each $\pi \in \Aut(G) - \{id \}$ if $C_1 \ldots C_s$ is the cycle decomposition of $\pi$, then there is an $i \in [s]$ such that $C_i$ contains two adjacent vertices.  Then, if $P_{\ell}(G,k) = P(G,k)$ for some $k \in \N$, $P_{\ell}(\mathcal{G},k) = P(\mathcal{G},k)$.
\end{pro}

\begin{proof}
Since we know that for each $\pi \in \Aut(G) - \{id \}$ if $C_1 \ldots C_s$ is the cycle decomposition of $\pi$, then there is an $i \in [s]$ such that $C_i$ contains two adjacent vertices, Theorem~\ref{thm: Hanlon} and Lemma~\ref{lem: reduce} implies $P(\mathcal{G},k) = P(G, id, k)/|\Aut(G)| = P(G,k)/|\Aut(G)|$.

Now, suppose that $L$ is a $k$-assignment for $G$ such that $u_{\ell}(G,L) = P_{\ell}(\mathcal{G},k)$.  Then,  Lemmas~\ref{lem: ineq} and~\ref{lem: reduce-list} and the fact that $P_{\ell}(G,k) = P(G,k)$ imply
$$ P_{\ell}(\mathcal{G},k) = u_{\ell}(G,L) \geq \frac{P(G, id, L)}{|\Aut(G)|} \geq \frac{P_{\ell}(G,k)}{|\Aut(G)|} = P(\mathcal{G},k).$$
The desired result immediately follows from the fact $P_{\ell}(\mathcal{G},k) \leq P(\mathcal{G},k)$.
\end{proof}

We end this section by considering disconnected graphs. In particular, we can show that the unlabeled list color function of a disconnected graph equals the product of the corresponding list color functions of its components as long as the components are non-isomorphic. 

\begin{custompro}{\bf\ref{pro: disjoint}}
    Suppose $G_1$ and $G_2$ are two non-isomorphic connected graphs. Let $G$ be the disjoint union of $G_1$ and $G_2$. Suppose $\mathcal{G}$, $\mathcal{G}_1$, and $\mathcal{G}_2$ are the unlabeled graphs corresponding to $G$, $G_1$ and $G_2$, respectively.
    Then, $P_{\ell}(\mathcal{G},k) = P_{\ell}(\mathcal{G}_1,k)\;P_{\ell}(\mathcal{G}_2,k)$.
\end{custompro}

\begin{proof}
    Since $G_1$ and $G_2$ are not isomorphic the function $F: \Aut(G_1) \times \Aut(G_2) \rightarrow \Aut(G)$ given by $F(\pi_1, \pi_2) = \pi$ where 
    \[  \pi(v) = \left\{
\begin{array}{ll}
      \pi_1(v) & \text{if $v\in V(G_1)$} \\
      \pi_2(v) & \text{if $v\in V(G_2)$} 
 \end{array} \right.  \]
is a bijection.

Now, suppose that $L$ is a $k$-assignment for $G$ such that $u_{\ell}(G,L) = P_{\ell}(\mathcal{G},k)$.  For each $i \in [2]$, let $L_i$ be the $k$-assignment for $G_i$ obtained by restricting the domain of $L$ to $V(G_i)$.  We will show that $u_{\ell}(G,L) = u_{\ell}(G_1,L_1) u_{\ell}(G_2,L_2)$ which will imply $P_{\ell}(\mathcal{G},k) \geq P_{\ell}(\mathcal{G}_1,k) P_{\ell}(\mathcal{G}_2,k)$.  Note that if for some $i \in [2]$, $G_i$ is not $L_i$-colorable, then $G$ is not $L$-colorable and our desired equation clearly holds.  So, we may suppose that $G_i$ is $L_i$-colorable for each $i \in [2]$.  Let $\mathcal{E}_i$ denote the set of equivalence classes counted by $u_{\ell}(G_i,L_i)$ for each $i \in [2]$, and let $\mathcal{E}$ denote the set of equivalence classes counted by $u_{\ell}(G,L)$.  We will prove the desired result by constructing a bijection $M: \mathcal{E}_1 \times \mathcal{E}_2 \rightarrow \mathcal{E}$.

Suppose $\mathcal{L}_i$ is the set of proper $L_i$-colorings of $G_i$ for each $i \in [2]$, and suppose $\mathcal{L}$ is the set of proper $L$-colorings of $G$.  Let $C: \mathcal{L}_1 \times \mathcal{L}_2 \rightarrow \mathcal{L}$ be the bijection given by $C(f_1,f_2) = f$ where $f$ is the proper $L$-coloring of $G$ formed by coloring the vertices in $G_i$ according to $f_i$ for each $i \in [2]$.  Now, suppose $M$ is given by
$$M(E_1,E_2) = \{C(f_1,f_2) : f_1 \in E_1, f_2 \in E_2 \}.$$
In order to prove $M$ is a function, we must show $M(E_1,E_2) \in \mathcal{E}$ whenever $E_i \in \mathcal{E}_i$ for each $i \in [2]$.  Suppose $g \in M(E_1,E_2)$, $g = C(g_1,g_2)$ where $g_i \in \mathcal{E}_i$ for each $i \in [2]$, and $E_g$ is the equivalence class in $\mathcal{E}$ containing $g$.  We claim $M(E_1,E_2) = E_g$ which will imply $M$ is a function.

Suppose $h \in M(E_1,E_2)$ is such that $h = C(h_1,h_2)$ where $h_i \in \mathcal{E}_i$ for each $i \in [2]$.  Now, since $h_i, g_i \in \mathcal{E}_i$ for each $i \in [2]$, there is a $\pi_i \in \Aut(G_i)$ such that $g_i \pi_i = h_i$ for each $i \in [2]$.  Consequently, $g F(\pi_1, \pi_2) = h$ which means that $g$ and $h$ are equivalent proper $L$-colorings of $G$ and $h \in E_g$.  Now, suppose $q \in E_g$.  Then, there is a $\sigma \in \Aut(G)$ such that $g \sigma = q$.  Suppose $(\sigma_1,\sigma_2) = F^{-1}(\sigma)$.  Then, $g_i \sigma_i \in E_i$ for each $i \in [2]$.  Also, if we let $q_i = g_i \sigma_i$ for each $i \in [2]$, then $q = C(q_1,q_2)$ which means $q \in M(E_1,E_2)$.  Thus, $M(E_1,E_2) = E_g$.

To see that $M$ is one-to-one suppose that $(A_1, A_2), (B_1,B_2) \in \mathcal{E}_1 \times \mathcal{E}_2$ satisfy $(A_1,A_2) \neq (B_1, B_2)$.  We will show that $M(A_1, A_2) \neq M(B_1,B_2)$ when $A_1 \neq B_1$ (the case where $A_1=B_1$ and $A_2 \neq B_2$ is similar).  Since $A_1$ and $B_1$ are different equivalence classes, they are nonempty disjoint sets.  Suppose $a_1 \in A_1$ and $a_2 \in A_2$.  Then $C(a_1,a_2) \in M(A_1,A_2)$, but since $C$ is a bijection and $a_1 \notin B_1$, we have $C(a_1,a_2) \notin M(B_1,B_2)$.  Thus, $M(A_1,A_2) \neq M(B_1,B_2)$.

To see that $M$ is onto, suppose that $E$ is an arbitrary element of $\mathcal{E}$ that contains the proper $L$-coloring $f$.  Suppose $C^{-1}(f) = (f_1,f_2)$ and $E_i$ is the element of $\mathcal{E}_i$ containing $f_i$ for each $i \in [2]$.  Now, by our proof that $M$ is a function, we know that since $ f \in M(E_1,E_2)$, $M(E_1,E_2)$ is the equivalence class in $\mathcal{E}$ containing $f$; that is, $M(E_1,E_2) = E$.  Thus, $M$ is a bijection.

Finally, to see that $P_{\ell}(\mathcal{G},k) \leq P_{\ell}(\mathcal{G}_1,k) P_{\ell}(\mathcal{G}_2,k)$,
suppose that for each $i \in [2]$, $\mathcal{K}_i$ is a $k$-assignment for $G_i$ such that $u_{\ell}(G_i,\mathcal{K}_i) = P_{\ell}(\mathcal{G}_i,k)$.  Suppose that $\mathcal{K}$ is the $k$-assignment for $G$ obtained by assigning each $v \in V(G) \cap V(G_i)$ the list $\mathcal{K}_i(v)$ for each $i \in [2]$.  The same argument as the argument employed above can be used to show $u_{\ell}(G,\mathcal{K}) = u_{\ell}(G_1,\mathcal{K}_1) u_{\ell}(G_2,\mathcal{K}_2)$ which implies $P_{\ell}(\mathcal{G},k) \leq u_{\ell}(G,K) = P_{\ell}(\mathcal{G}_1,k) P_{\ell}(\mathcal{G}_2,k)$.  Thus, our proof is complete.  \end{proof}

\section{Point-Determining Graphs and Beyond}\label{proof}

In this section our first priority is to prove Theorem~\ref{thm: usefenmingnminus2}.  Before we give the proof, we need a lemma.

\begin{lem} \label{lem: sumbound}
Suppose $G = P_n$ with $n \geq 2$.  Suppose that the vertices of $G$ in order are $v_1, \ldots, v_n$.  Suppose $L$ is a $k$-assignment for $G$, and let $s = \sum_{i=2}^n |L(v_i) - L(v_{i-1})|$.  If $s < k$, then $|L(v_1) \cap L(v_n)| \geq k-s$.
\end{lem}

\begin{proof}
The proof is by induction on $n$.  Note that the result is clear when $n=2$.  So, suppose that $n > 2$ and the desired result holds for all integers greater than 1 and less than $n$.

By the induction hypothesis we have that $|L(v_1) \cap L(v_{n-1})| \geq k - (s - |L(v_n) - L(v_{n-1})|)$ which implies $|L(v_1) \cap L(v_{n-1})|- |L(v_n) - L(v_{n-1})| \geq k - s$ which implies $|L(v_1) \cap L(v_{n-1})| + |L(v_{n-1}) \cap L(v_{n})| - k \geq k-s$.

Now, notice that of the elements in $(L(v_{n-1}) \cap L(v_{n}))$ at most $|L(v_{n-1})-L(v_1)| = k - |L(v_1) \cap L(v_{n-1})|$ are not in $L(v_1)$.  Consequently, $|L(v_1) \cap L(v_n)| \geq |L(v_{n-1}) \cap L(v_{n})| - (k - |L(v_1) \cap L(v_{n-1})|) = |L(v_1) \cap L(v_{n-1})| + |L(v_{n-1}) \cap L(v_{n})| - k$.  The desired result follows.
\end{proof}

We are now ready to prove Theorem~\ref{thm: usefenmingnminus2}.

\begin{customthm} {\bf\ref{thm: usefenmingnminus2}}
Suppose $\mathcal{G}$ is an unlabeled, connected graph of order $n$ and size $m \geq 4$ with $G \in \mathcal{G}$.  Suppose for each $\pi \in \Aut(G) - \{id \}$ with cycle decomposition $C_1 \ldots C_s$ we have: If for each $i \in [s]$ the vertices in $C_i$ are pairwise nonadjacent in $G$, $s \leq n-2$.  Then, there exists an $N \in \N$ such that $P_{\ell}(\mathcal{G},k) = P(\mathcal{G},k)$ whenever $k \geq N$.
\end{customthm}

\begin{proof}
Let $A$ be the set of all $\pi \in \Aut(G)$ such that if $C_1 \ldots C_s$ is the cycle decomposition of $\pi$, then for each $i \in [s]$ the vertices in $C_i$ are pairwise nonadjacent in $G$.  Let $B$ be the subset of $A$ that consists of all the elements of $A$ with exactly $n-2$ cycles in its cycle decomposition.  Also, let $|A|=a$ and $|B|=b$.  By Theorem~\ref{thm: Hanlon} and Lemma~\ref{lem: reduce}, we have that for each $k \in \N$,
$$ P(\mathcal{G},k) = \frac{1}{|\Aut(G)|} \sum_{\pi \in \Aut(G)} P(G, \pi, k) =\frac{1}{|\Aut(G)|} \sum_{\pi \in A} P(G : \pi, k).$$
The hypotheses imply that for each $\pi \in A - (\{id\} \cup B)$ and $k \in \N$, $P(G : \pi, k) \leq k^{n-3}$, and for each $\pi \in B$ and $k \in \N$, $P(G : \pi, k) \leq k^{n-2}$. Consequently,
$$ P(\mathcal{G},k) \leq \frac{1}{|\Aut(G)|} \left( P(G,k) + bk^{n-2} + (a-b)k^{n-3}  \right).$$

Now, for some fixed $k \in \N$ satisfying $k \geq m-1$ suppose that $L$ is a $k$-assignment for $G$. Let $s = \sum_{uv\in E(G)}|L(u) - L(v)|$. Then by Lemma~\ref{lem: ineq} and Theorem~\ref{thm: fenming},
\begin{align*}
u_{\ell}(G,L) &\geq \frac{1}{|\Aut(G)|} P(G,L) \\
&\geq \frac{1}{|\Aut(G)|} \left( P(G,k) + s(k - m + 1)k^{n-3} \right). 
\end{align*}
So, if $s \geq b+1$ we obtain:
$$P(\mathcal{G},k) - u_{\ell}(G,L) \leq \frac{1}{|\Aut(G)|}  \left(-k^{n-2} + ((a-b)+(m-1)(b+1))k^{n-3}   \right).$$
So, when $s \geq b+1$ and $k > (a-b) + (m - 1)(b+1)$, we have that $u_{\ell}(G,L) > P(\mathcal{G},k)$.   

Now, suppose that $0< s \leq b$ and $k > 2b + n - 3$.  Then by Lemma~\ref{lem: ineq} and Theorem~\ref{thm: fenming},
\begin{align*}
u_{\ell}(G,L) &\geq \frac{1}{|\Aut(G)|} \left( P(G,L) + \sum_{\pi \in B} P(G, \pi, L) \right) \\
&\geq \frac{1}{|\Aut(G)|} \left( P(G,k) + s(k - m + 1)k^{n-3} + \sum_{\pi \in B} P(G, \pi, L) \right). 
\end{align*}
Using the notation of Lemma~\ref{lem: reduce-list}, we know that for $\pi \in B$, $P(G, \pi, L) = P(G : \pi, L')$.  We also have that $G : \pi$ is an $n-2$ vertex graph, and the cycle decomposition of $\pi$ consists of $(n-4)$ 1-cycles and two $2$-cycles, or the cycle decomposition of $\pi$ consists of $(n-3)$ 1-cycles and one $3$-cycle.  Since $G$ is connected, Lemma~\ref{lem: sumbound} implies that for any $u,v \in V(G)$, $|L(u) \cap L(v)| \geq k-s$ which also implies for any $u,v,w \in V(G)$, $|L(u) \cap L(v) \cap L(w)| \geq k-2s$.  Consequently, we have that $|L'(z)| \geq k-2s \geq k - 2b$ for each $z \in V(G: \pi)$, and $P(G, \pi, L) = P(G : \pi, L') \geq \prod_{i=0}^{n-3} (k-2b-i)$.  Let $p = \prod_{i=0}^{n-3} (k-2b-i) - k^{n-2}$.  Since $k > 2b+n-3$, we have $0 \leq (2b+i)/k < 1$ for each $i \in \{0,\ldots,n-3\}$.  Consequently,
\begin{align*}
\prod_{i=0}^{n-3}(k-2b-i)
=k^{n-2}\prod_{i=0}^{n-3}\left(1-\frac{2b+i}{k}\right) &\geq k^{n-2}\left(1-\sum_{i=0}^{n-3}\frac{2b+i}{k}\right)\\
&=k^{n-2}-(n-2)\left(2b+\frac{n-3}{2}\right)k^{n-3}.
\end{align*}
Thus, $p \geq -(n-2)\left(2b+(n-3)/2\right)k^{n-3}$.  Now, we have that
$$P(\mathcal{G},k) - u_{\ell}(G,L) \leq \frac{1}{|\Aut(G)|}  \left(-k^{n-2} + \left((a-b)+(m-1) -bp/k^{n-3}\right)k^{n-3}   \right).$$
Consequently,
$$P(\mathcal{G},k) - u_{\ell}(G,L) \leq \frac{1}{|\Aut(G)|}  \left(-k^{n-2} + \left((a-b)+(m-1)+b(n-2)\left(2b+\frac{n-3}{2}\right)\right)k^{n-3}   \right).$$
So, when $0 < s \leq b$, $k > 2b + n - 3$, and $k > (a-b) + (m - 1)+b(n-2)\left(2b+(n-3)/2\right)$, we have that $u_{\ell}(G,L) > P(\mathcal{G},k)$.

Since $s > 0$ if and only if $L$ doesn't assign the same list to every vertex of $G$, $P_{\ell}(\mathcal{G},k) = P(\mathcal{G},k)$ whenever
$$k > \max\left\{(a-b) + (m - 1)(b+1),2b+n-3,(a-b) + (m - 1)+b(n-2)\left(2b+\frac{n-3}{2}\right)\right\}.$$

\end{proof}

A graph $G$ is called \emph{point-determining} if no two distinct vertices in $G$ have the same neighborhood in $G$ (see e.g.~\cite{GL11, S73}).  It is easy to see that $G$ is point-determining if and only if all transpositions in $\Aut(G)$ interchange two adjacent vertices. In order to complete the proof of Corollary~\ref{cor: pointdetermine}, we have to verify the conclusion for connected point-determining graphs with at most 3 edges. The only such graphs are $K_1, K_2, K_3, P_4$.  If $G$ is a copy of any of these graphs the following result applies.

\begin{pro} [\cite{KS90}] \label{pro: chordal}
If $G$ is a chordal graph, then $P_{\ell}(G,k) = P(G,k)$ whenever $k \in \N$.
\end{pro}

So, when $G$ is a copy of $K_1, K_2, K_3$, or $P_4$ it is easy to see that the hypotheses of Proposition~\ref{pro: oneaut} are satisfied.

\begin{customcor}  {\bf\ref{cor: pointdetermine}}
If $\mathcal{G}$ is an unlabeled, connected, point-determining graph, then there exists an $N \in \N$ such that $P_{\ell}(\mathcal{G},k) = P(\mathcal{G},k)$ whenever $k \geq N$.
\end{customcor}

We end the section with the proof of Theorem~\ref{thm: transposition}.

\begin{customthm} {\bf\ref{thm: transposition}}
Suppose $\mathcal{G}$ is an unlabeled graph of order $n$ with $n$ possibly zero. Suppose $G \in \mathcal{G}$.  Suppose that $G$ is point-determining and for each $\pi \in \Aut(G) - \{id\}$ if $C_1 \ldots C_s$ is the cycle decomposition of $\pi$, then there is an $i \in [s]$ such that $C_i$ contains at least two adjacent vertices. 

Let $G'$ be the graph obtained from $G$ by adding nonadjacent vertices $x$ and $y$, and edges so that $x$ and $y$ are adjacent to every element of $V(G)$.  Also, let $\mathcal{G}'$ be the unlabeled graph such that $G' \in \mathcal{G}'$.  Then, there exists an $N \in \N$ such that $P_{\ell}(\mathcal{G}',k) = P(\mathcal{G}',k)$ whenever $k \geq N$.
\end{customthm}

\begin{proof}
 First, notice that $|\Aut(G')| = 2|\Aut(G)|$.  Additionally, the hypotheses imply that the only elements of $\Aut(G')$ that have no cycle in their cyclic decomposition with two adjacent vertices in $G'$ are the identity permutation and the transposition interchanging $x$ and $y$.  So, if $G''$ is the graph obtained from $G'$ by identifying $x$ and $y$ as the same vertex, Theorem~\ref{thm: Hanlon} and Lemma~\ref{lem: reduce} imply that for each $k \geq 2$ 
 \begin{align*}
     P(\mathcal{G}',k) &= \frac{1}{2|\Aut(G)|} \left(P(G',k) + P(G'',k) \right) \\ 
     &= \frac{1}{2|\Aut(G)|} \left(k(k-1)P(G,k-2) + 2P(G'',k) \right) \\
     &= \frac{1}{2|\Aut(G)|} \left(k(k-1)P(G,k-2) + 2kP(G,k-1) \right).
 \end{align*}
 Now, for some fixed $k \in \N$ satisfying $k \geq |E(G)|+3$ suppose that $L$ is a $k$-assignment for $G'$.  Let $C = L(x) \cap L(y)$ and $c = |C|$.  Note that $0 \leq c \leq k$, and we know from Section~\ref{intro} that $P_{\ell}(G,q) = P(G,q)$ whenever $q \geq k-2$.

 Clearly, each proper $L$-coloring of $G'$ is equivalent to at most $2|\Aut(G)|$ proper $L$-colorings of $G'$. However, if $f$ is a proper $L$-coloring of $G'$ that colors $x$ and $y$ with the same color from $C$, then $f$ can be equivalent to at most $|\Aut(G)|$ proper $L$-colorings of $G'$.  This is because for each $\pi \in \Aut(G')$ that has both $x$ and $y$ as fixed points, $f \pi = f (xy) \pi$ where $(xy)$ is the transposition interchanging $x$ and $y$.

 Also, if $f$ is a proper $L$-coloring of $G'$ that doesn't color $x$ with an element from $C$, then $f$ can be equivalent to at most $|\Aut(G)|$ proper $L$-colorings.  This is because for each $\pi \in \Aut(G')$ containing the cycle $(xy)$, $f \pi$ is not a proper $L$-coloring of $G'$.  Similarly, if $f$ is a proper $L$-coloring of $G'$ that doesn't color $y$ with an element from $C$, then $f$ can be equivalent to at most $|\Aut(G)|$ proper $L$-colorings.  All these facts imply 
\begin{align*} u_{\ell}(G',L) &\geq \frac{c(c-1)P_{\ell}(G,k-2)}{2|\Aut(G)|} + \frac{cP_{\ell}(G,k-1)}{|\Aut(G)|} + \frac{(k^2 - c^2)P_{\ell}(G,k-2)}{|\Aut(G)|} \\
&= \frac{c(c-1)P(G,k-2)}{2|\Aut(G)|} + \frac{cP(G,k-1)}{|\Aut(G)|} + \frac{(k^2 - c^2)P(G,k-2)}{|\Aut(G)|}.
\end{align*}
So, we have that
$$P(\mathcal{G}',k) - u_{\ell}(G',L) \leq \frac{1}{2|\Aut(G)|}  \left((k-c) (2P(G,k-1) - (c+k+1) P(G,k-2))   \right).$$
Since 
$$\lim_{k \rightarrow \infty} \frac{P(G,k-1)}{P(G,k-2)} = 1,$$  there is an $N_G \in \N$ such that $P(G,k-1)/P(G,k-2) \leq (k+1)/2$ whenever $k \geq N_G$.  So, we have that if $k \geq N_G$, $P(\mathcal{G}',k) - u_{\ell}(G',L) \leq 0$.  It immediately follows that $P_{\ell}(G',k) = P(\mathcal{G}',k)$ whenever $k \geq \max\{N_G,|E(G)|+3\}$. 
\end{proof}

\vspace*{0.5cm}
\noindent{\bf Acknowledgement.} The authors thank Prof. Bruce Sagan for encouraging and helpful conversations.

\paragraph{Declaration of generative AI and AI-assisted technologies in the manuscript preparation process.}
During the preparation of this work, the authors used ChatGPT (OpenAI) for limited language editing and for checking algebraic computations.  After using this tool, the authors reviewed and edited the content as needed and take full responsibility for the content of the publication.

\end{document}